\documentclass[leqno,12pt]{article} 
\usepackage{amsmath, amssymb}
\usepackage{amsthm} 
\usepackage{amssymb}
\usepackage{mathrsfs}
\usepackage{amssymb, url, color, graphicx, amscd, mathrsfs}
\usepackage[colorlinks=true, bookmarks=true, pdfstartview=FitH, pagebackref=true, linktocpage=true, linkcolor = magenta, citecolor = blue]{hyperref}
\usepackage{graphicx}
\usepackage{times}
\theoremstyle{plain} 
\newtheorem{theorem}{\indent\sc Theorem}[section]
\newtheorem{lemma}[theorem]{\indent\sc Lemma}
\newtheorem{corollary}[theorem]{\indent\sc Corollary}

\theoremstyle{definition} 

\numberwithin{equation}{section}
\DeclareMathOperator{\Ric}{Ric}

\makeatletter
\def\address#1#2{\begingroup
\noindent\parbox[t]{7.8cm}{%
\small{\scshape\ignorespaces#1}\par\vskip1ex
\noindent\small{\itshape E-mail address}%
\/: #2\par\vskip4ex}\hfill%
\endgroup}%
\makeatother
\title{{Rigidity and vanishing theorems for complete translating solitons}} 
\author{
\textsc{Ha Tuan Dung, Nguyen Thac Dung, and Tran Quang Huy} 
}
\date{} 
\usepackage[pagewise]{lineno}
\begin{document}

\maketitle



\begin{abstract}
In this paper, we prove some rigidity theorems for complete translating solitons. Assume that the $L^q$-norm of the trace-free second fundamental form is finite, for some $q\in\mathbb{R}$ and using a Sobolev inequality, we show that translator must be hyperspace. Our results can be considered as a generalization of \cite{Ma, WXZ16, Xin15}. We also investigate a vanishing property for translators which states that there are no nontrivial $L_f^p\ (p\geq2)$ weighted harmonic $1$-forms on ${M}$ if the $L^n$-norm of the second fundamental form is bounded.
\end{abstract}
\section{Introduction}
Let $X_0:M^n \rightarrow \mathbb{R}^{n+m}$ be an $n$-dimensional smooth submanifold immersed in an $(n+m)$-dimensional Euclidean space $\mathbb{R}^{n+m}$. The mean curvature flow with initial value $X_0$ is a family of immersions $X: M \times [0,T) \rightarrow \mathbb{R}^{n+m}$ satisfying 
\begin{equation}
\left\{\begin{array}{ll}
\frac{d}{d t} X(x,t)=H(x,t), \\
X(x,0)=X_0(x),
\end{array}\right.
\end{equation}
for $x \in M, t \in [0,T)$, where $H(x,t)$ is the mean curvature vector of $X_t(M)=M_t$ at $X(x,t)$ in $\mathbb{R}^{n+m}$ and  $X_t(\cdot)=X(t,\cdot)$.

 One of the most important parts in study of mean curvature flow is singularity analysis. For several situations, the second fundamental form with respect to the family $M_t$ may develop singularities. For example, if $M$ is compact, the second fundamental form will blow up in a finite time. According to the blow-up rate of the second fundamental form, we divide singularities of mean curvature flow into two types, Type-I singularity and Type-II singularity. Since the geometry of the solution near the Type-II singularity cannot be controlled well, the study of Type-II singularities is much more complicated than type-I singularities. 

A very important example of Type-II singularities is the translating soliton. A submanifold $X: M^n \rightarrow \mathbb{R}^{n+m}$ is said to be a translating soliton if there exists a unit constant vector $V$ in $\mathbb{R}^{n+m}$ such that
\begin{equation}
H=V^N,
\end{equation}
where $V^N$ denotes the normal component of $V$ in $\mathbb{R}^{n+m}$. Let $V^T$ be the tangent component of $V$, then 
\begin{equation}\label{translatingequa}
H+V^T = V.
\end{equation}

Translating solitons often occur as Type-II singularities after rescaling. Also, a translating soliton corresponds to a translating solution $M_t$ of the mean curvature flow defined by $M_t=M+tV$. There are very few examples of translating solitons even in the hypersurface case. The first basic examples are translating solitons which are also the minimal hypersurfaces. By \eqref{translatingequa} we know that $V$ must be tangential to the soliton. Consequently, these solitons could have the form of $\tilde{M} \times L$, where $L$ is a line parallel to $V$ and $\tilde{M}$ is a minimal hypersurface in $L^{\perp}$. Some other non-trivial examples are grim reapers and grim reaper cylinders. The grim reaper $\Gamma$ is a one-dimensional translating soliton in $\mathbb{R}^2$ defined as the graph of
$$y=-\log \cos x, \quad x \in\left(-\frac{\pi}{2}, \frac{\pi}{2}\right).$$
A trivial generalization is the Euclidean product $\Gamma \times \mathbb{R}^{n-1}$ in $\mathbb{R}^{n+1}$, which is called grim reaper cylinder. For further examples of translators, we refer the reader to \cite{Nguyen22, Nguyen23}.
 
As mentioned above, translators play an important role in the study of mean curvature flow because they arise as blow-up solutions at type II singularities. Huisken and Sinestrari \cite{HS99} proved that at type II singularity of a mean curvature flow with mean convex solution, there exists a blow-up solution which is a convex translating solution. In \cite{Wang11}, Wang proved that when $n=2$, every entire convex translator must be rotationally
symmetric. However, in every dimension greater than two, there exist non-rotationally symmetric, entire convex translators. Later, Haslhofer \cite{Has} obtained the uniqueness theorem
of the bowl soliton in all dimensions under the assumptions of uniformly 2-convexity and
noncollapsing condition. In \cite{Xin15}, Xin studied some basic properties of translating solitons: the volume growth, generalized maximum principle, Gauss maps and certain functions related to the Gauss map. He also gave integral estimates for the squared norm of the second fundamental form and used them to show rigidity theorems for translators in the Euclidean space in higher co-dimension. Some of Xin's results then were extended by Wang, Xu, and Zhao by using integral curvature pinching conditions of the trace-free second fundamental form (see \cite{WXZ16}). Recently, in \cite{IR17, IR19} Impera and Rimoldi studied rigidity results and topology at infinitity of translating solitons of the mean curvature flows. Their approaches relies on the theory of $f$-minimal hypersufaces. In particular, they established weighted Sobolev inequalities and used them to show that an $f$-stable translator has at most one end.  They also investigate some relationship between the space of $L^2$ weighted harmonic $1$-forms, cohomology with compact support, and the index of the translator in term of generalized Morse index of stable operator.  Using Sobolev inequalities discovered by Impera and Rimoldi, Kunikawa and Sato \cite{KS19} pointed out that any complete $f$-stable translating soliton admits no codimension one cycle which does not disconnect $M$. As a consequence, any two dimensional complete $f$-stable translator has genus zero. For further discussion on translators, we refer to \cite{Guang19, Has, HS99, IR17, IR19, KP17, KS19, Nguyen22, Nguyen23, Smoczyk, WXZ16, Wang11, Xin15} and the references therein.

In this paper, motivated by \cite{IR17, Ma, WXZ16, Xin15}, we investigate some rigidity theorems and study the connectedness at infinity of complete translators in Euclidean spaces. Our first main theorem is as follows
\begin{theorem}\label{th:th12}
Let $M^n (n \geq 3)$ be a smooth complete translating soliton in
the Euclidean space $\mathbb{R}^{n+m}$. If the trace-free second fundamental form ${\Phi}$ of $M$
satisfies
$$\left(\int_{M}|{\Phi}|^{n} \mathrm{d} \mu\right)^{1 / n}<K(n,a) \quad \text {and}\quad \int_{M}|{\Phi}|^{2a} e^{(V, X)}<\infty,$$
where  $$1\leq a < \dfrac{n+\ \sqrt[]{n^2-2n}}{2},$$
$$K(n,a)=\sqrt{\frac{{{(n-2)}^{2}}\left( a-\frac{1}{2} \right)}{{{D}^{2}}(n)\left[ \frac{{{(n-2)}^{2}}\left( a-\frac{1}{2} \right)}{2n\left( na-\frac{n}{2}-{{a}^{2}} \right)}+{{(n-1)}^{2}} \right]\iota {{a}^{2}}}},\quad \iota=\begin{cases}2&\text{if }m=1,\\
4&\text{if }m\geq2 \end{cases},$$  and $D(n)$ is the Sobolev constant defined in Lemma \ref{1:1}. Then $M$ is a
linear subspace.
\end{theorem}
The proof of this theorem relies on a Sobolev inequality in immersed submanifolds which was first verified in \cite{HS74} and \cite{MS73}. When $a=\frac{n}{2}$, Theorem \ref{th:th12} recovers Theorem 1 in \cite{WXZ16}. As noted in \cite{WXZ16}, the curvature condition in Theorem \ref{th:th12} is weaker than that in Theorem 7.1 in \cite{Xin15}. If translators are located in a halfspace, in \cite{IR17, IR19}, Impera and Rimoldi proved a weighted Sobolev inequality by using the bijective corresponding found out by Smoczyk \cite{Smoczyk} between translators and minimal hypersurfaces in a suitable warped product. Apply their Sobolev inequality, we are able to obtain the following theorem.
\begin{theorem}\label{main2}
Let $M^n (n \geq 3)$ be a smooth complete translating soliton in
the Euclidean space $\mathbb{R}^{n+1}$ contained in the halfspace $\Pi_{v, a}=$ $\left\{y \in \mathbb{R}^{n+1}:\langle y, v\rangle \geq a\right\}$, where $a\in\mathbb{R}, v\in\mathbb{R}^{n+1}$ are fixed. If the second fundamental form $A$ of $M$
satisfies
$$\left(\int_{M} |A|^{n} \varrho \right)^{\frac{1}{n}}<\sqrt{\frac{(n^2-2n+2)(n-2)^2}{n^3 S(n)^2 (n-1)^2}},$$
	where $S(n)$ is the Sobolev constant as in Lemma \ref{sob} and $\rho=e^{\left\langle V, X \right\rangle}$.
 Then $M$ is a
hyperplane.
\end{theorem}
Compared with the results in \cite{WXZ16} \cite{Xin15}, this result drops the assumption on the smallness of $L^n$-norm of $|A|$, instead of this, we require the weighed $L^n$-norm of $|A|$ is small. In fact, in \cite{Xin15}, the author supposed that the weighted $L^n$-norm of $|A|$ is finite and $L^n$-norm is small.
Hence, when the weighted $L^n$-norm of $|A|$ is small, this theorem can be considered as a refinement of Theorem 7.1 in \cite{Xin15}. Moreover, using the weighted Sobolev inequality, we obtain a vanishing theorem as follows.
\begin{theorem}\label{harmonic}
Let $x: {M}^{n} \rightarrow \mathbb{R}^{n+1}$ be a smooth complete translator contained in the halfspace $\Pi_{v,a}=\{p\in\mathbb{R}^{n+1}: \left\langle p,v\right\rangle\geq a\}$, where $v\in\mathbb{R}^{n+1}$ is a given vector, $a\in\mathbb{R}$ is fixed, and $n\geq3$. Assume that for any $p\geq2,$
$${{\left\| A \right\|}_{n,f}} < \frac{\sqrt{\left( p-1 \right)}\left( n-1 \right)}{pS(n)},$$
	where $S(n)$ is the Sobolev constant as in Lemma \ref{sob}.
Then there are no nontrivial $L_f^p$ weighted harmonic $1$-forms on ${M}.$
\end{theorem}

This paper has four sections. Section 1 is used to derive some rigidity theorems. Then we prove a vanishing result for weighted harmonic forms in Section 3. Finally, we study translators in the Euclidean space with a Sobolev inequality in Section 4 and give another rigidity theorem. 
\section{Rigidity theorems}
 \quad Let $X:M \rightarrow \mathbb{R}^{n+m}$ be an $n$-dimensional translating soliton. $H,A,{\Phi}$ are the mean curvature vector, second fundamental form, trace-free second fundamental form of $M$, respectively. $V$ is the constant vector so that $V^N=H$. Let $f=-{\langle V, X\rangle}$, we define 
$$\begin{aligned}
\Delta_f&=\Delta+\langle V, \nabla(\cdot)\rangle=e^{-(V, X)} \operatorname{div}\left(e^{\langle V, X\rangle} \nabla(\cdot)\right)\\
&=e^{f}\operatorname{div}(e^{-f}\nabla(\cdot)).
\end{aligned}$$ 
Denote the trace-free second fundamental form $\Phi=A-\frac{1}{n}g\otimes H$. It is well-known that
$$|\Phi|^2=|A|^2-\frac{1}{n}|H|^2 \quad{\text{ and }}\quad |\nabla\Phi|^2=|\nabla A|^2-\frac{1}{n}\nabla|H|^2.$$
In order to prove our theorems, we need the following Simons type identity which has been obtained by Xin \cite{Xin15} (see also \cite{WXZ16}).
\begin{lemma}[\cite{IR17, WXZ16}]
 On a translating soliton $M^{n}$ in $\mathbb{R}^{n+m},$ we have
\begin{equation}\label{eq:eq1}
\Delta_f|{\Phi}|^{2} \geq 2|\nabla| {\Phi}||^{2}-\iota|{\Phi}|^{4}-\frac{2}{n}|H|^{2}|{\Phi}|^{2},
\end{equation}
where
$$
\iota=\left\{\begin{array}{ll}
2, & \text { if } \quad m=1 \\
4, & \text { if } \quad m \geq 2.
\end{array}\right.
$$
Moreover, when $m=1$, we have
\begin{equation}\label{e02}
\Delta_f|\Phi|^2=2|\nabla\Phi|^2-2|A|^2|\Phi|^2.
\end{equation} 
\end{lemma}
We now recall that the following Sobolev inequality for submanifolds in the Euclidean is very helpful to derive our rigidity theorems (see \cite{XG07}).
\begin{lemma}[Sobolev's inequality]\label{1:1}
 Let $M^{n}\,(n \geq 3)$ be a complete submanifold in the Euclidean space $\mathbb{R}^{n+m} .$ Let $f$ be a nonnegative $C^{1}$ function with compact support. Then for all $s \in \mathbb{R}^{+},$ we have
$$
\|f\|_{\frac{2 n}{n-2}}^{2} \leq D^{2}(n)\left[\frac{4(n-1)^{2}(1+s)}{(n-2)^{2}}\|\nabla f\|_{2}^{2}+\left(1+\frac{1}{s}\right) \frac{1}{n^{2}}\||H| f\|_{2}^{2}\right],
$$
where $$D(n)=2^{n}(1+n)^{\frac{n+1}{n}}(n-1)^{-1} \sigma_{n}^{-\frac{1}{n}},$$ and $\sigma_{n}$ denotes the volume of the
unit ball in $\mathbb{R}^{n}$.
\end{lemma}

For the proof of Theorem \ref{th:th12}, let $\rho=e^{\left\langle V, X\right\rangle}$, we can follow, verbatim,
the proof in \cite{WXZ16}, but using, instead of the function $f$ defined in Lemma 5 of
\cite{WXZ16}, the function $\varphi=|\Phi|^a\rho^{1/2}\eta$, where $a\geq1$ is a constant which possible values
will be determined later and $\eta$ is a smooth function
with compact support on $M$. For the convenience of the reader, in order to help him/her
checking the influence of the constant $a$ in every step, we give all the details of the computations.
\begin{lemma}
Assume that $|\Phi|\not=0$ on $M$. If $\eta$ be a smooth function
with compact support on $M$, then
\begin{equation}\label{eq:eq2}
\int_{M}|\nabla \varphi|^{2}=\int_{M}\left|\nabla\left(|{\Phi}|^{a} \eta\right)\right|^{2} \varrho-\frac{1}{2} \int_{M}|{\Phi}|^{2a} \eta^{2} \varrho+\frac{1}{4} \int_{M}|{\Phi}|^{2a}\left|V^{T}\right|^{2} \eta^{2} \varrho.
\end{equation}
\end{lemma}
\begin{proof}
Integrating by parts, we have 
$$\begin{aligned} \int_{M}|\nabla \varphi|^{2} &=\int_{M}\left|\nabla\left(|{\Phi}|^{a} \eta\right)\right|^{2} \varrho+\frac{1}{2} \int_{M} \nabla\left(|{\Phi}|^{2a} \eta^{2}\right) \nabla \varrho+\int_{M}|{\Phi}|^{2a} \eta^{2}\left|\nabla \varrho^{\frac{1}{2}}\right|^{2} \\ &=\int_{M}\left|\nabla\left(|{\Phi}|^{a} \eta\right)\right|^{2} \varrho-\frac{1}{2} \int_{M}|{\Phi}|^{2a} \eta^{2} \Delta \varrho+\int_{M}|{\Phi}|^{2a} \eta^{2}\left|\nabla \varrho^{\frac{1}{2}}\right|^{2}. \end{aligned}$$
By $M^n$ is a translating soliton, we have 
$$\nabla \varrho=\nabla e^{\langle V, X\rangle}=\varrho V^{T},$$
$$\nabla \varrho^{\frac{1}{2}}=\frac{1}{2} \varrho^{-\frac{1}{2}} \nabla \varrho=\frac{1}{2} \varrho^{\frac{1}{2}} V^{T},$$
and
$$\Delta \varrho=\sum_{i} \nabla_{i} \varrho\left\langle V, e_{i}\right\rangle+\sum_{i} \varrho\left\langle V, \nabla_{i} e_{i}\right\rangle=\varrho\left(\left|V^{T}\right|^{2}+\left|H\right|^{2}\right)=\varrho\left(\left|V^{T}\right|^{2}+\left|V^{N}\right|^{2}\right)=\varrho.$$
Using them, we get
$$\int_{M}|\nabla \varphi|^{2}=\int_{M}\left|\nabla\left(|{\Phi}|^{a} \eta\right)\right|^{2} \varrho-\frac{1}{2} \int_{M}|{\Phi}|^{2a} \eta^{2} \varrho+\frac{1}{4} \int_{M}|{\Phi}|^{2a}\left|V^{T}\right|^{2} \eta^{2} \varrho.$$
\end{proof}
Now, combining the Sobolev's inequality in Lemma \ref{1:1} and \eqref{eq:eq2}, we get 
$$\begin{aligned}
&\left(\int_{M}|\varphi|^{\frac{2 n}{n-2}}\right)^{\frac{n-2}{n}} \\
&\leq  D^{2}(n) \cdot\left\{\frac{4(n-1)^{2}(1+s)}{(n-2)^{2}} \int_{M}|\nabla \varphi|^{2}+\left(1+\frac{1}{s}\right) \cdot \frac{1}{n^{2}} \int_{M}|H|^{2} \varphi^{2}\right\} \\
&= D^{2}(n) \cdot\left\{\frac{4(n-1)^{2}(1+s)}{(n-2)^{2}}\left(\int_{M}\left|\nabla\left(|{\Phi}|^{a} \eta\right)\right|^{2} \varrho-\frac{1}{2} \int_{M}|{\Phi}|^{2a} \eta^{2} \varrho\right.\right.\\
&\quad\left.\left.+\frac{1}{4} \int_{M}|{\Phi}|^{2a}\left|V^{T}\right|^{2} \eta^{2} \varrho\right)+\left(1+\frac{1}{s}\right) \cdot \frac{1}{n^{2}} \int_{M}|{\Phi}|^{2a}|H|^{2} \eta^{2} \varrho\right\}.
\end{aligned}$$
Note that
$$|V^T|^2 + |V^N|^2=|V^T|^2+|H|^2=1.$$
Thus, we obtain
\begin{equation} 
\begin{aligned}
&\left(\int_{M}|\varphi|^{\frac{2 n}{n-2}}\right)^{\frac{n-2}{n}} \\
&\leq  D^{2}(n) \cdot\left\{\frac{4(n-1)^{2}(1+s)}{(n-2)^{2}}\left(\int_{M}\left|\nabla\left(|{\Phi}|^{a} \eta\right)\right|^{2} \varrho-\frac{1}{4} \int_{M}|{\Phi}|^{2a}\left|V^{T}\right|^{2} \eta^{2} \varrho\right.\right.\\
&\quad\left.\left.-\frac{1}{2} \int_{M}|{\Phi}|^{2a}|H|^{2} \eta^{2} \varrho\right)+\left(1+\frac{1}{s}\right) \cdot \frac{1}{n^{2}} \int_{M}|{\Phi}|^{2a}|H|^{2} \eta^{2} \varrho\right\} \\
&= D^{2}(n) \cdot\left\{\frac{4(n-1)^{2}(1+s)}{(n-2)^{2}}\left(\int_{M} a^2|\nabla| {\Phi}||^{2}|{\Phi}|^{2a-2} \eta^{2} \varrho\right.\right.\\
&\quad+\int_{M} 2a|{\Phi}|^{2a-1} \eta \nabla|{\Phi}| \cdot \nabla \eta \varrho+\int_{M}|{\Phi}|^{2a}|\nabla \eta|^{2} \varrho-\frac{1}{4} \int_{M}|{\Phi}|^{2a}\left|V^{T}\right|^{2} \eta^{2} \varrho \\
&\quad\left.\left.-\frac{1}{2} \int_{M}|{\Phi}|^{2a}|H|^{2} \eta^{2} \varrho\right)+\left(1+\frac{1}{s}\right) \cdot \frac{1}{n^{2}} \int_{M}|{\Phi}|^{2a}|H|^{2} \eta^{2} \varrho\right\}.
\end{aligned}
\end{equation}
By the Cauchy inequality, for $\delta >0$ we have 
\begin{equation}\label{eq:eq5}
\begin{aligned}
&\left(\int_{M}|\varphi|^{\frac{2 n}{n-2}}\right)^{\frac{n-2}{n}} \\
&\leq  \frac{4 D^{2}(n)(n-1)^{2}(1+s)}{(n-2)^{2}}\left\{(1+\delta) a^2 \int_{M}|\nabla| {\Phi} \|^{2}|{\Phi}|^{2a-2} \eta^{2} \varrho\right.\\
&\quad\left.+\left(1+\frac{1}{\delta}\right) \int_{M}|{\Phi}|^{2a}|\nabla \eta|^{2} \varrho-\frac{1}{4} \int_{M}|{\Phi}|^{2a}\left|V^{T}\right|^{2} \eta^{2} \varrho-\frac{1}{2} \int_{M}|{\Phi}|^{2a}|H|^{2} \eta^{2} \varrho\right\} \\
&\quad+D^{2}(n)\left(1+\frac{1}{s}\right) \cdot \frac{1}{n^{2}} \int_{M}|{\Phi}|^{2a}|H|^{2} \eta^{2} \varrho.
\end{aligned}
\end{equation}
In order to estimate the term $\int_{M} |\nabla| {\Phi}||^{2}|{\Phi}|^{2a-2} \eta^{2} \varrho$, we multiply $|{\Phi}|^{2a-2} \eta^{2}$ on both sides of \eqref{eq:eq1} and integrating by parts
with respect to the measure $\varrho d\mu$ on $M$,
\begin{equation}\label{eq:eq3}
\begin{aligned}
0 &\geq 2 \int_{M}|\nabla| {\Phi}||^{2}|{\Phi}|^{2a-2} \eta^{2} \varrho-\iota \int_{M}|{\Phi}|^{2a+2} \eta^{2} \varrho-\frac{2}{n} \int_{M}|{\Phi}|^{2a}|H|^{2} \eta^{2} \varrho \\
&\quad-\int_{M}|{\Phi}|^{2a-2} \eta^{2} \Delta_f|{\Phi}|^{2} \varrho.
\end{aligned}
\end{equation}
Since $\eta$ has compact support on $M$, by the Stokes theorem, we obtain
\begin{equation}\label{eq:eq4}
\begin{aligned}
&-\int_{M}|{\Phi}|^{2a-2} \eta^{2} \Delta_f|{\Phi}|^{2} \varrho \\
&=-\int_{M}|{\Phi}|^{2a-2} \eta^{2} \operatorname{div}\left(\varrho \cdot \nabla|{\Phi}|^{2}\right) \\
&= 2 \int_{M} \varrho|{\Phi}| \nabla|{\Phi}| \cdot \nabla\left(|{\Phi}|^{2a-2} \eta^{2}\right) \\
&= 4(a-1) \int_{M}|\nabla| {\Phi} \|^{2}|{\Phi}|^{2a-2} \eta^{2} \varrho+4 \int_{M}(\nabla|{\Phi}| \cdot \nabla \eta)|{\Phi}|^{2a-1} \eta \varrho.
\end{aligned}
\end{equation}
Combining \eqref{eq:eq3} and \eqref{eq:eq4}, we get 
$$\begin{aligned}
0 &\geq  4\left(a-\frac{1}{2}\right) \int_{M}|\nabla| {\Phi}||^{2}|{\Phi}|^{2a-2} \eta^{2} \varrho-\iota \int_{M}|{\Phi}|^{2a+2} \eta^{2} \varrho-\frac{2}{n} \int_{M}|{\Phi}|^{2a}|H|^{2} \eta^{2} \varrho \\
&\quad+4 \int_{M}(\nabla|{\Phi}| \cdot \nabla \eta)|{\Phi}|^{2a-1} \eta \varrho.
\end{aligned}$$
By the Cauchy inequality, for $0< \varepsilon < a-\frac{1}{2},$ we have 
\begin{equation}\label{eq:eq6}
\begin{aligned}
\iota \int_{M}|{\Phi}|^{2a+2} \eta^{2} \varrho+\frac{2}{n} \int_{M}|{\Phi}|^{2a}|H|^{2} \eta^{2} \varrho+\frac{1}{\varepsilon} \int_{M}|{\Phi}|^{2a}|\nabla \eta|^{2} \varrho \\
\geq 4\left(a-\frac{1}{2}-\varepsilon\right) \int_{M}|\nabla| {\Phi}||^{2}|{\Phi}|^{2a-2} \eta^{2} \varrho.
\end{aligned}
\end{equation}
Substituting \eqref{eq:eq6} into \eqref{eq:eq5}, we get
\begin{equation}
\begin{aligned}
&\left(\int_{M}|\varphi|^{\frac{2 n}{n-2}}\right)^{\frac{n-2}{n}} \\
 &\leq \frac{4 D^{2}(n)(n-1)^{2}(1+s)}{(n-2)^{2}}\left\{\frac{a^2(1+\delta)}{4\left(a-\frac{1}{2}-\varepsilon\right)}\left(\iota \int_{M}|{\Phi}|^{2a+2} \eta^{2} \varrho\right.\right.\\
&\quad\left.+\frac{2}{n} \int_{M}|{\Phi}|^{2a}|H|^{2} \eta^{2} \varrho+\frac{1}{\varepsilon} \int_{M}|{\Phi}|^{2a}|\nabla \eta|^{2} \varrho\right) \\
&\quad\left.+\left(1+\frac{1}{\delta}\right) \int_{M}|{\Phi}|^{2a}|\nabla \eta|^{2} \varrho-\frac{1}{2} \int_{M}|{\Phi}|^{2a}|H|^{2} \eta^{2} \varrho\right\} \\
&\quad+D^{2}(n)\left(1+\frac{1}{s}\right) \cdot \frac{1}{n^{2}} \int_{M}|{\Phi}|^{2a}|H|^{2} \eta^{2} \varrho.
\end{aligned}
\end{equation}
We want to get rid of the term $\int_{M}|{\Phi}|^{2a}|H|^{2} \eta^{2} \varrho$ by choosing $\delta>0$ appropriately.
Put 
$$\delta= \delta(n,\varepsilon,a)= \frac{\left(2(n-1)^2 n^2 s - (n-2)^2 \right)\left(a-\frac{1}{2}-\varepsilon \right)}{2(n-1)^2a^2ns}-1.$$
We would require $\delta >0$, this occurs only if $s$ satisfies 
\begin{equation}\label{eq:eq8}
s>\frac{(n-2)^2\left(a-\frac{1}{2}-\varepsilon\right)}{2(n-1)^2 n \left(na-\frac{n}{2}-a^2-n\epsilon \right)}
\end{equation}
for some $\varepsilon \in \left(0,a-\frac{1}{2}-\frac{a^2}{n} \right)$ defined later and also, we need $$1\leq a<\frac{n+\sqrt{n^2-2n}}{2}.$$ Consequently, we have
\begin{equation}\label{eq:eq7}
\begin{aligned}
& \kappa^{-1}\left(\int_{M}|\varphi|^{\frac{2 n}{n-2}}\right)^{\frac{n-2}{n}} \\
 &\leq \frac{a^{2}(1+s)(1+\delta)}{4\left(a-\frac{1}{2}-\varepsilon\right)}\left(\iota \int_{M}|{\Phi}|^{2a+2} \eta^{2} \varrho+\frac{1}{\varepsilon} \int_{M}|{\Phi}|^{2a}|\nabla \eta|^{2} \varrho\right) \\
&\quad+(1+s)\left(1+\frac{1}{\delta}\right) \int_{M}|{\Phi}|^{2a}|\nabla \eta|^{2} \varrho \\
&= \frac{(1+s) \iota \left[2 s n^{2}(n-1)^{2}-(n-2)^{2}\right]}{8 s n(n-1)^{2}} \int_{M}|{\Phi}|^{2a+2} \eta^{2} \varrho \\
&\quad+C(s, \varepsilon, n,a) \int_{M}|{\Phi}|^{2a}|\nabla \eta|^{2} \varrho,
\end{aligned}
\end{equation}
where $C(s,\varepsilon,n,a)$ is an explicit positive constant depending on $s,\varepsilon,n,a$ and 
$$\kappa=\frac{4D^2(n)(n-1)^2}{(n-2)^2}.$$
By the H\"{o}lder inequality, we have 
\begin{equation}
\begin{aligned}
\int_{M}|{\Phi}|^{2a+2} \eta^{2} \varrho 
& \leq\left(\int_{M}(|{\Phi}|^{2\cdot\frac{n}{2}}\right)^{\frac{2}{n}} \cdot\left(\int_{M}\left(|{\Phi}|^{2a} \eta^{2} \varrho\right)^{\frac{n}{n-2}}\right)^{\frac{n-2}{n}}\\
 & = \left(\int_{M}|{\Phi}|^{2 \cdot \frac{n}{2}}\right)^{\frac{2}{n}} \cdot\left(\int_{M}\left|\varphi\right|^{\frac{2n}{n-2}}\right)^{\frac{n-2}{n}}.
\end{aligned}
\end{equation}
Applying this to \eqref{eq:eq7}, we have 
\begin{equation}\label{eq:eq9}
\begin{aligned}
& \kappa^{-1}\left(\int_{M}|\varphi|^{\frac{2 n}{n-2}}\right)^{\frac{n-2}{n}} \\
&\leq  \frac{(1+s) \iota\left[2 s n^{2}(n-1)^{2}-(n-2)^{2}\right]}{8 s n(n-1)^{2}}\left(\int_{M}|{\Phi}|^{2a}\right)^{\frac{2}{n}} \cdot\left(\int_{M}|\varphi|^{\frac{2 n}{n-2}}\right)^{\frac{n-2}{n}} \\
&\quad+C(s, \varepsilon, n,a) \int_{M}|{\Phi}|^{2a}|\nabla \eta|^{2} \varrho.
\end{aligned}
\end{equation}
Put 
$$K(n, s)=\sqrt{\frac{8 s n(n-1)^{2}}{(1+s) \iota\left[2 s n^{2}(n-1)^{2}-(n-2)^{2}\right] \kappa}}.$$
By the condition \eqref{eq:eq8}, we can choose
$$s=\frac{(n-2)^2\left(a-\frac{1}{2}\right)}{2(n-1)^2 n \left(na-\frac{n}{2}-a^2-n\varepsilon \right)}.$$
Hence, substituting $s$ into $K(n,s)$, we have
\begin{equation}
K(n,a,\varepsilon )=K(n,s(a,\varepsilon ))=\sqrt{\frac{{{(n-2)}^{2}}\left( a-\frac{1}{2} \right)}{{{D}^{2}}(n)\left[ \frac{{{(n-2)}^{2}}\left( a-\frac{1}{2} \right)}{2n\left( na-\frac{n}{2}-{{a}^{2}}-n\varepsilon  \right)}+{{(n-1)}^{2}} \right]\iota (n\varepsilon +{{a}^{2}})}}.
\end{equation}
Set
$$K(n,a)=\underset{\varepsilon \in \left( 0,\frac{(n-a-1)(a-1)}{n} \right)}{\mathop{\sup }}\,K(n,a,\varepsilon )=\sqrt{\frac{{{(n-2)}^{2}}\left( a-\frac{1}{2} \right)}{{{D}^{2}}(n)\left[ \frac{{{(n-2)}^{2}}\left( a-\frac{1}{2} \right)}{2n\left( na-\frac{n}{2}-{{a}^{2}} \right)}+{{(n-1)}^{2}} \right]\iota {{a}^{2}}}}.$$
We now can give a proof of Theorem \ref{th:th12}.
\begin{proof}[Proof of Theorem \ref{th:th12}]
Since we have the assumption
$\left(\int_{M}|{\Phi}|^{n} \mathrm{d} \mu\right)^{1 / n}<K(n,a),$
there exists a positive constant $\check{K}$ such that
\begin{equation}\label{eq:eq10}
\left(\int_{M}|{\Phi}|^{n} \mathrm{d} \mu\right)^{1 / n}<\check{K}<K(n,a).
\end{equation}
Thus, there exists $\varepsilon = \varepsilon_0 > 0$ such that
$$\check{K}<K\left(n, a,\varepsilon_{0}\right)<K(n,a).$$
Using this and combining \eqref{eq:eq9}, \eqref{eq:eq10}, there exists $0<\epsilon<1$ such that 
\begin{equation}
\begin{aligned}
& \kappa^{-1}\left(\int_{M}|f|^{\frac{2 n}{n-2}}\right)^{\frac{n-2}{n}} \\
&\leq  \kappa^{-1} \cdot K\left(n,a, \varepsilon_{0}\right)^{-2} \cdot \check{K}^{2}\left(\int_{M}|f|^{\frac{2 n}{n-2}}\right)^{\frac{n-2}{n}}+\bar{C}\left(n, a,\varepsilon_{0}\right) \int_{M}|{\Phi}|^{2a}|\nabla \eta|^{2} \varrho \\
&\leq  \frac{1-\epsilon}{\kappa}\left(\int_{M}|f|^{\frac{2 n}{n-2}}\right)^{\frac{n-2}{n}}+\bar{C}\left(n,a, \varepsilon_{0} \right) \int_{M}|{\Phi}|^{2a}|\nabla \eta|^{2} \varrho
\end{aligned}
\end{equation}
or equivalently 
\begin{equation}\label{eq:eq11}
\frac{\epsilon}{\kappa}\left(\int_{M}|f|^{\frac{2 n}{n-2}}\right)^{\frac{n-2}{n}} \leq \bar{C}\left(n,a, \varepsilon_{0}\right) \int_{M}|{\Phi}|^{2a}|\nabla \eta|^{2} \varrho.
\end{equation}
Let $$\eta(X)=\eta_{r}(X)=\phi\left(\frac{|X| }{r}\right)$$ for any $r>0$, where $\phi$ is a non-negative smooth function on $[0,+\infty)$ satisfying 
\begin{equation}
\phi(x)=\left\{\begin{array}{ll}
1, & \text { if } \quad x \in[0,1) \\
0, & \text { if } \quad x \in[2,+\infty)
\end{array}\right.
\end{equation} 
and $|\phi'| \leq C$ for some absolute constant. Since $\int_{M}|{\Phi}|^{q} \varrho$ and $\bar{C}(n,a,\varepsilon_0)$ are bounded, then the right hand side of \eqref{eq:eq11} approaches to zero as $r \to \infty$, which implies the left hand side equal to zero or namely, $|{\Phi}| \equiv 0$. 

Finally, using the assertion that $|\Phi|=0$, it was confirmed in \cite{WXZ16} that $M$ is a linear subspace. In the rest of the proof, we give a detail arguments which is inspired by Impera and Rimoldi in \cite{IR17}. We argue as follows. Since $|\Phi|=0$, it turns out that $|A|^2=\frac{1}{n}H^2$. Moreover, we note that $|\nabla \Phi|^2=|\nabla|\Phi||^2=0$. This implies $$0=|\nabla\Phi|^2=|\nabla A|^2-\frac{1}{n}|\nabla H|^2.$$ Therefore, we get
\begin{align}
\left| \nabla {{\left| A \right|}^{2}} \right|=\frac{1}{n}\left| \nabla {{H}^{2}} \right|&=\frac{2}{n}\left| H \right|\left| \nabla H \right|\nonumber\\
&=\frac{2}{n}\left( \sqrt{n}\left| A \right| \right)\left( \sqrt{n}\left| \nabla A \right| \right)=2\left| A \right|\left| \nabla A \right|.\nonumber
\end{align}
As a consequence,  $|\nabla A|=|\nabla|A||$. Therefore, we can apply the argument in the proof of Theorem A in \cite{IR17} to conclude that $M$ is a linear subspace. The proof is complete.
\end{proof}
Observe that $[1, n-1]\subset \left[1, \frac{n+\ \sqrt[]{n^2-2n}}{2}\right)$, the weighted $L^{2a}$ norm of $|\Phi|$ in our theorem is wider than those in \cite{WXZ16}. Moreover, when $a=\frac{n}{2}$, our theorem recovers the following rigidity property which was obtained by Wang \textit{et. all} in \cite{WXZ16}.
\begin{theorem}
Let $M^n (n \geq 3)$ be a smooth complete translating soliton in
the Euclidean space $\mathbb{R}^{n+m}$. If the trace-free second fundamental form ${\Phi}$ of $M$
satisfies
$$\left(\int_{M}|{\Phi}|^{n} \mathrm{d} \mu\right)^{1 / n}<K(n) \quad {\text{and}}\quad \int_{M}|{\Phi}|^{n} e^{(V, X)}<\infty$$
where $K(n)$ is defined as above, then $M$ is a
linear subspace.
\end{theorem}
It is worth to mention that the above condition is weaker than that in the rigidity theorem of Xin \cite{Xin15}. Now, to derive another rigidity result,  we can use the following version of Sobolev inequality. 
\begin{lemma}[\cite{IR17}]\label{sob}
Let $X: M^{n} \rightarrow \mathbb{R}^{n+1}$ be a translator contained in the halfspace $$\Pi_{v, a}=\left\{p \in \mathbb{R}^{n+1}:\langle p, v\rangle \geq a\right\}$$ for some $a \in \mathbb{R}$ and $n\geq3$. Let $u$ be a non-negative compactly supported $C^{1}$ function on $M$. Then
\begin{equation}\label{eq:eq17}
\left[\int_{M} u^{\frac{2n}{n-2}} \varrho d \mu\right]^{\frac{n-2}{n}} \leq \left(\frac{2(n-1)S(n)}{n-2}\right)^2 \int_{M}|\nabla u|^2 \varrho d \mu
\end{equation}
where $S(n)$ is the Sobolev constant given in Lemma 4.2 in \cite{IR17}.
\end{lemma}
Repeating the same computation as above, we can give a verification of Theorem \ref{main2} as follows. 
\begin{proof}[Proof of Theorem \ref{main2}]
Applying the Sobolev's inequalities \eqref{eq:eq17} to $u=|{\Phi}|^{\frac{n}{2}} \eta$ and using the Cauchy inequality, we have
$$\begin{aligned}
\left[\int_{M} \left(|{\Phi}|^{\frac{n}{2}} \eta \right)^{\frac{2n}{n-2}} \rho \right]^{\frac{n-2}{n}} 
 &\leq \left(\frac{2 S(n)( n-1)}{n-1}\right)^{2} \int_{M}\left|\nabla  \left(|{\Phi}|^\frac{n}{2} \eta \right)\right|^{2} \rho \\
 &=\left(\frac{2 S(n) (n-1)}{n-2}\right)^{2}\left(\int_{M} \frac{n^2}{4}|\nabla| {\Phi}||^{2}|{\Phi}|^{n-2} \eta^{2} \varrho\right.\\
&\left. \quad +\int_{M} n|{\Phi}|^{n-1} \eta \nabla|{\Phi}| \cdot \nabla \eta \varrho+\int_{M}|{\Phi}|^{n}|\nabla \eta|^{2} \varrho \right)\\
&\leq  \left(\frac{2 S(n) (n-1)}{n-2}\right)^{2}\left((1+\delta)\int_{M} \frac{n^2}{4}|\nabla| {\Phi}||^{2}|{\Phi}|^{n-2} \eta^{2} \varrho\right.\\
&\left.\quad+\left(1+\frac{1}{\delta} \right) \int_{M}|{\Phi}|^{n}|\nabla \eta|^{2} \varrho \right).
\end{aligned}$$
Applying \eqref{eq:eq6} and notice that $\iota=2$, we have, for $0<\varepsilon< \frac{n}{2}-\frac{1}{2}$ 
\begin{equation}\label{eq:eq18}
 \begin{aligned}
\kappa_{2}^{-1} \left[\int_{M} \left(|{\Phi}|^{\frac{n}{2}} \eta \right)^{\frac{2(n)}{n-2}} \rho \right]^{\frac{n-2}{n}} &\leq \left\{\frac{\frac{n^2}{4}(1+\delta)}{4\left(\frac{n}{2}-\frac{1}{2}-\varepsilon\right)}\left(2 \int_{M}|{\Phi}|^{n+2} \eta^{2} \varrho\right.\right.\\
&\quad\left.+\frac{2}{n} \int_{M}|{\Phi}|^{n}|H|^{2} \eta^{2} \varrho+\frac{1}{\varepsilon} \int_{M}|{\Phi}|^{n}|\nabla \eta|^{2} \varrho\right) \\
&\quad\left.+\left(1+\frac{1}{\delta}\right) \int_{M}|{\Phi}|^{n}|\nabla \eta|^{2} \varrho\right\},
\end{aligned}
 \end{equation}
where $\kappa_2=\left(\frac{2S(n)(n-1)}{n-2}\right)^2$. 
By the fact that $|A|^2=|{\Phi}|^2+\frac{1}{n}|H|^2$, we can rewrite \eqref{eq:eq18} as 
$$\begin{aligned}
\kappa_{2}^{-1} \left[\int_{M} \left(|{\Phi}|^{\frac{n}{2}} \eta \right)^{\frac{2 n}{n-2}} \rho \right]^{\frac{n-2}{n}} &\leq \frac{\frac{n^2}{4}(1+\delta)}{2\left(\frac{n}{2}-\frac{1}{2}-\varepsilon\right)}\int_{M}|{\Phi}|^{n} |A|^2 \eta^{2} \varrho\\
&\quad+\hat{C}(n,\delta,\varepsilon) \int_{M}|{\Phi}|^{n}|\nabla \eta|^{2} \varrho ,
\end{aligned}$$
where $\hat{C}(n,\delta,\varepsilon)$ is explicit positive constant depending on $n,\delta,\varepsilon$. Applying H\"{o}lder inequality, we have
 $$\begin{aligned}
\kappa_{2}^{-1} \left[\int_{M} \left(|{\Phi}|^{\frac{n}{2}} \eta \right)^{\frac{2n}{n-2}} \varrho \right]^{\frac{n-2}{n}} 
&\leq \frac{n^2(1+\delta)}{8\left(\frac{n}{2}-\frac{1}{2}-\varepsilon\right)} \left(\int_{M}(|A|^{2}\rho^{2/n})^{n/2}\right)^{\frac{2}{n}} \cdot\left(\int_{M}\left(|{\Phi}|^{\frac{n}{2}} \eta \varrho^{\frac{n-2}{2n}} \right)^{\frac{2n}{n-2}}\right)^{\frac{n-2}{n}}\\
&\quad+\hat{C}(n,\delta,\varepsilon) \int_{M}|{\Phi}|^{n}|\nabla \eta|^{2} \varrho ,\\
&\leq \frac{n^2(1+\delta)}{8\left(\frac{n}{2}-\frac{1}{2}-\varepsilon\right)} \left(\int_{M}|A|^{n}\varrho\right)^{\frac{2}{n}} \cdot\left(\int_{M}\left(|{\Phi}|^{\frac{n}{2}} \eta \varrho \right)^{\frac{2n}{n-2}} \right)^{\frac{n-2}{n}}\\
&\quad+ \hat{C}(n,\delta,\varepsilon) \int_{M}|A|^n\nabla \eta|^{2} \varrho ,\\
\end{aligned}$$
here we used $|\Phi|\leq |A|$ in the last inequality. 
Put $$K_2(n,\varepsilon,\delta) = \sqrt{\frac{8\left(\frac{n}{2}-\frac{1}{2}-\varepsilon\right)}{n^2(1+\delta)\kappa_2}},$$
and 
$$K_2(n) = \sup_{\delta>0, \,0< \varepsilon <a-\frac{n-1}{n}} K_2(n,\varepsilon,\delta) = \sqrt{\frac{(n-2)^2}{S(n)^2(n-1) n^2}}.$$
By the assumption
$$\left(\int_{M} |A|^{n} \varrho \right)^{\frac{1}{n}}<K_2(n)$$
and using the same argument as Theorem \ref{th:th12}, we complete the proof. 
\end{proof}
Now, as mentioned in \cite{IR17}, an application of the maximum principle and the weighted version of a result
in \cite{CS80} give that translators with mean curvature that does not change sign are
either $f$-stable (generalizing in particular Theorem 1.2.5 in \cite{42}, and Theorem 2.5 in \cite{421}) or they split as the
product of a line parallel to the translating direction and a minimal hypersurface
in the orthogonal complement of the line. Note that, in this latter case, by Fubini's 
theorem, the condition $|A|\in L^p({M}_f)$ for some $p>0$  is met if and only if $|A|\equiv 0$ (i.e. ${M}$ is a translator hyperplane). Moreover, to adapt the ideas in \cite{SSY} for minimal surface , Ma and Miquel proved in \cite{Ma} a refined Kato inequality on translating solitons as follows.
\begin{lemma}[\cite{Ma}]\label{ma}
Let $M^{n}$ be a hypersurface immersed in $\mathbb{R}^{n+1}$ satisfying 
$$|\nabla A|\leq \frac{3n+1}{2n}|\nabla H|,$$
then we have
$$|\nabla\Phi|^2\geq\frac{n+1}{n}|\nabla|\Phi||^2.$$
\end{lemma}
Note that on the translating soliton $M$, we have $\nabla H=\left\langle \nabla\nu, v\right\rangle=A(\cdot, v)$, so the condition becomes $|\nabla A|\leq \frac{3n+1}{2n}|A(\cdot, v)|$. Now, under these assumptions, we obtain the following result which can be considered as an improvement of  Theorem 6 in \cite{Ma}.
\begin{theorem}\label{dung1}
Let $x: {M}^{n\ge 2} \rightarrow \mathbb{R}^{n+1}$ be a translator with mean curvature which
does not change sign. Suppose that
$$|\nabla A|\leq \frac{3n+1}{2n}|\nabla H|$$
and the traceless second fundamental form of the
immersion satisfies $|\Phi| \in L^{p}\left({M}_{f}\right)$ for $p\in \left( 2-\frac{2}{\sqrt[]{n}},2+\frac{2}{\sqrt[]{n}} \right).$ Then ${M}$ is a translator hyperplane.
\end{theorem}
\begin{proof}Since the curvature does not change sign, we may assume that $M$ is $f$-stable. Otherwise, $|A|\equiv0$, so $M$ is a hyperplane. From the definition of the $f$-Laplacian operator and the equation \eqref{e02}, we have 
$$\left| \Phi  \right|{{\Delta }_{f}}\left| \Phi  \right|=|\nabla \Phi|^{2}-{{\left| \nabla \left| \Phi  \right| \right|}^{2}}-|{A}|^{2}|{\Phi}|^{2}.$$
By the Kato-type inequality in Lemma \ref{ma}, this implies 
\begin{align}\label{02}
\left| \Phi  \right|{{\Delta }_{f}}\left| \Phi  \right|\ge \frac{1}{n}{{\left| \nabla \left| \Phi  \right| \right|}^{2}}-|A{{|}^{2}}|\Phi {{|}^{2}}.
\end{align}
Now, let $\eta$ be a smooth compactly supported function on ${M}.$
For any $a>1,$ multiplying $\left| \Phi  \right|^{a-1} \eta^2$ both sides of the \eqref{02} and integrating by parts with respect to the measure $e^{-f} d\mu$ on ${M}$ yield
\begin{align}\label{03}
\int_{{M} }{{{\eta^2\left| \Phi  \right|}^{a}}{{\Delta }_{f}}\left| \Phi  \right|{{e}^{-f}}d\mu }\ge \frac{1}{n}\int_{{M} }{{{\eta^2\left| \Phi  \right|}^{a-1}}{{\left| \nabla \left| \Phi  \right| \right|}^{2}}{{e}^{-f}}d\mu }-\int_{{M} }{|A{{|}^{2}}\eta^2|\Phi {{|}^{a+1}}{{e}^{-f}}d\mu }.
\end{align}
Since $\eta$ has compact support on ${M},$ by the Stokes theorem, it shows that 
\begin{align*}
\int_{{M} }{{{\eta }^{2}}{{\left| \Phi  \right|}^{a}}{{\Delta }_{f}}\left| \Phi  \right|{{e}^{-f}}d\mu }&=-\int_{{M} }{\left\langle \nabla \left( {{\eta }^{2}}{{\left| \Phi  \right|}^{a}} \right),\nabla \left| \Phi  \right| \right\rangle {{e}^{-f}}d\mu }\nonumber\\
&=-\int_{{M} }{\left\langle 2\eta {{\left| \Phi  \right|}^{a}}\nabla \eta +a{{\eta }^{2}}{{\left| \Phi  \right|}^{a-1}}\nabla \left| \Phi  \right|,\nabla \left| \Phi  \right| \right\rangle {{e}^{-f}}d\mu }\nonumber\\
&=-2\int_{{M} }{ {{\left| \Phi  \right|}^{a}}\left\langle \nabla \eta ,\nabla \left| \Phi  \right| \right\rangle \eta {{e}^{-f}}d\mu }-a\int_{{M} }{{{\eta }^{2}}{{\left| \Phi  \right|}^{a-1}}{{\left| \nabla \left| \Phi  \right| \right|}^{2}}{{e}^{-f}}d\mu }.
\end{align*}
Subsituting the above identity into \eqref{03}, we obtain 
\begin{align}
&-2\int_{{M} }{{{\left| \Phi  \right|}^{a}}\left\langle \nabla \eta ,\nabla \left| \Phi  \right| \right\rangle \eta  {{e}^{-f}}d\mu }-a\int_{{M} }{{{\eta }^{2}}{{\left| \Phi  \right|}^{a-1}}{{\left| \nabla \left| \Phi  \right| \right|}^{2}}{{e}^{-f}}d\mu }\nonumber\\
&\quad\ge \frac{1}{n}\int_{{M} }{{{\eta^2\left| \Phi  \right|}^{a-1}}{{\left| \nabla \left| \Phi  \right| \right|}^{2}}{{e}^{-f}}d\mu }-\int_{{M} }{|A{{|}^{2}}\eta^2|\Phi {{|}^{a+1}}{{e}^{-f}}d\mu },\nonumber
\end{align}
or equivalently 
\begin{align}\label{05}
&\left( a+\frac{1}{n} \right)\int_{{M} }{{{\eta }^{2}}{{\left| \Phi  \right|}^{a-1}}{{\left| \nabla \left| \Phi  \right| \right|}^{2}}{{e}^{-f}}d\mu }\nonumber\\
&\quad\quad\le 2\int_{{M} }{ {{\left| \Phi  \right|}^{a}}\left\langle \nabla \eta ,\nabla \left| \Phi  \right| \right\rangle \eta {{e}^{-f}}d\mu }\ +\int_{{M} }{|A{{|}^{2}}\eta^2|\Phi {{|}^{a+1}}{{e}^{-f}}d\mu }.
\end{align}
On the other hand, since ${M}$ satisfies the stablity inequality, we have 
$$\int_{{M} }{{{\left| A \right|}^{2}}{{\psi }^{2}}{{e}^{-f}}d\mu }\le \int_{{M} }{{{\left| \nabla \psi  \right|}^{2}}{{e}^{-f}}d\mu }.$$
Replacing $\psi$ by $\eta |\Phi {{|}^{\frac{a+1}{2}}}$  in the above inequality gives
\begin{align}\label{06}
	\int_{{M} }{|A{{|}^{2}}{{\eta }^{2}}|\Phi {{|}^{a+1}}{{e}^{-f}}d\mu }\le& \int_{{M} }{{{\left| \nabla \left( \eta |\Phi {{|}^{\frac{a+1}{2}}} \right) \right|}^{2}}{{e}^{-f}}d\mu }\nonumber\\
=&
 \int_{{M} }{{{\left| \Phi  \right|}^{a+1}}{{\left| \nabla \eta  \right|}^{2}}{{e}^{-f}}d\mu }+\left( a+1 \right)\int_{{M} }{|\Phi {{|}^{a}}\left\langle \nabla \eta ,\nabla |\Phi | \right\rangle \eta {{e}^{-f}}d\mu }\nonumber\\
&\quad+\frac{{{\left( a+1 \right)}^{2}}}{4}\int_{{M} }{|\Phi {{|}^{a-1}}{{\left| \nabla |\nabla \Phi | \right|}^{2}}{{\eta }^{2}}{{e}^{-f}}d\mu }.
\end{align}
Combining \eqref{05} and \eqref{06}, we have
\begin{align}
\left( a+\frac{1}{n} \right)\int_{{M} }{{{\eta }^{2}}{{\left| \Phi  \right|}^{a-1}}{{\left| \nabla \left| \Phi  \right| \right|}^{2}}{{e}^{-f}}d\mu }&\le 2\int_{{M} }{|\Phi {{|}^{a}}\left\langle \nabla \eta ,\nabla |\Phi | \right\rangle \eta {{e}^{-f}}d\mu }+\int_{{M} }{{{\left| \Phi  \right|}^{a+1}}{{\left| \nabla \eta  \right|}^{2}}{{e}^{-f}}d\mu }\nonumber\\
&\quad+\left( a+1 \right)\int_{{M} }{|\Phi {{|}^{a}}\left\langle \nabla \eta ,\nabla |\Phi | \right\rangle \eta {{e}^{-f}}d\mu }\nonumber\\
&\quad+\frac{{{\left( a+1 \right)}^{2}}}{4}\int_{{M} }{|\Phi {{|}^{a-1}}{{\left| \nabla |\nabla \Phi | \right|}^{2}}{{\eta }^{2}}{{e}^{-f}}d\mu }.\nonumber
\end{align}
Hence,
\begin{align}\label{010}
\left[ a+\frac{1}{n}-\frac{{{\left( a+1 \right)}^{2}}}{4} \right]\int_{{M} }{{{\eta }^{2}}{{\left| \Phi  \right|}^{a-1}}{{\left| \nabla \left| \Phi  \right| \right|}^{2}}{{e}^{-f}}d\mu }&\le \int_{{M} }{{{\left| \Phi  \right|}^{a+1}}{{\left| \nabla \eta  \right|}^{2}}{{e}^{-f}}d\mu }\nonumber\\
&\quad+\left( a+3 \right)\int_{{M} }{|\Phi {{|}^{a}}\left\langle \nabla \eta ,\nabla |\Phi | \right\rangle \eta {{e}^{-f}}d\mu }.
\end{align}
From the Cauchy-Schwarz inequality and the inequality $xy\le \varepsilon {{x}^{2}}+\frac{1}{4\varepsilon }{{y}^{2}}$ for all $\varepsilon>0,$ we see that
\begin{align}\label{011}
(a+3)|\Phi {{|}^{a}}\left\langle \nabla \eta ,\nabla |\Phi | \right\rangle \eta 
&\le |a+3||\Phi {{|}^{a}}\left| \nabla \eta  \right|\left| \nabla |\Phi | \right|\left| \eta  \right|\nonumber\\
&=|a+3|\left( |\Phi {{|}^{\frac{a-1}{2}}}\left| \nabla |\Phi | \right|\left| \eta  \right| \right)\left( |\Phi {{|}^{\frac{a+1}{2}}}\left| \nabla \eta  \right| \right)\nonumber\\
&\le\varepsilon|\Phi {{|}^{a-1}}{{\left| \nabla |\Phi | \right|}^{2}}{{\eta }^{2}}+\frac{\left( a+3 \right)^2}{4\varepsilon}|\Phi {{|}^{a+1}}{{\left| \nabla \eta  \right|}^{2}}.
\end{align}
Substituting \eqref{011} into \eqref{010}, we get
\begin{align}
&\left[ a+\frac{1}{n}-\frac{{{\left( a+1 \right)}^{2}}}{4}-\varepsilon \right]\int_{{M} }{{{\left| \Phi  \right|}^{a-1}}{{\left| \nabla \left| \Phi  \right| \right|}^{2}}{{\eta }^{2}}{{e}^{-f}}d\mu }\le \left[ 1+\frac{{{\left( a+3 \right)}^{2}}}{4\varepsilon} \right]\int_{{M} }{{{\left| \Phi  \right|}^{a+1}}{{\left| \nabla \eta  \right|}^{2}}{{e}^{-f}}d\mu }.\nonumber
\end{align}
Now let $p=a+1$. Then, the above inequality becomes
$$\left[ p-1-\frac{{{p}^{2}}}{4}+\frac{1}{n} -\varepsilon\right]\int_{{M} }{{{\left| \Phi  \right|}^{p-2}}{{\left| \nabla \left| \Phi   \right| \right|}^{2}}{{\eta }^{2}}{{e}^{-f}}d\mu }\le \left[ 1+\frac{{{\left( p+2 \right)}^{2}}}{4\varepsilon} \right]\int_{{M} }{{{\left| \Phi   \right|}^{p}}{{\left| \nabla \eta  \right|}^{2}}{{e}^{-f}}d\mu }.$$
Next, we choose the number $p$ to be
$$p-1-\frac{{{p}^{2}}}{4}+\frac{1}{n}>0,$$
or equivalently
 $$2-\frac{2}{\sqrt[]{n}}<p<2+\frac{2}{\sqrt[]{n}}=2\left( 1+\sqrt[]{\frac{1}{n}} \right).$$
Hence, for $2-\frac{2}{\sqrt[]{n}}< p < 2+\frac{2}{\sqrt[]{n}}$, we can choose $\varepsilon>0$ small such that there is a constant $C>0$ depending on $n, p$ such that 
$$\int_{{M} }{{{\left| \Phi  \right|}^{p-2}}{{\left| \nabla \left| \Phi   \right| \right|}^{2}}{{\eta }^{2}}{{e}^{-f}}d\mu }\le C\int_{{M} }{{{\left| \Phi   \right|}^{p}}{{\left| \nabla \eta  \right|}^{2}}{{e}^{-f}}d\mu }.$$
Now, for some fixed point $o\in M$ and $R>0$, we choose a test function $\eta$ satisfying $\eta\in\mathcal{C}^\infty(M), 0\le \eta\le 1$, $\eta=1$ in $B_o(R)$, $\eta=0$ outside $B_o(2R)$, and $|\nabla\eta|\le\frac{2}{R}$. Plugging $\eta$ in the above inequality then let $R$ tends to infinity, we conclude that $|\nabla|\Phi||=0$, since $|\Phi|\in L^p({M}_f)$. Therefore, $|\Phi|$ is constant. Note that a translating solition is of Euclidean growth volume (\cite{Xin15}), this implies $\Phi=0$ because $|\Phi|\in L^p({M}_f)$. Now, we apply the argument as in the proof of Theorem \ref{th:th12}   
 to conclude that ${M}$ is a hyperplane.
\end{proof}
As a consequence of this theorem, for $p=2$ we obtain the following corollary which can be considered as an improvement of Theorem 6 by Ma and Miquel in \cite{Ma}. 
\begin{corollary}
Let $x: {M}^{n\ge 2} \rightarrow \mathbb{R}^{n+1}$ be a translator with mean curvature which does not change sign and
$$|\nabla A|\leq \frac{3n+1}{2n}|\nabla H|.$$ Suppose that the traceless second fundamental form of the
immersion satisfies $|\Phi| \in L^{2}\left({M}_{f}\right)$. Then ${M}$ is a translator hyperplane.
\end{corollary}

\section{Vanishing result for weighted harmonic forms}
In this section we give a proof of Theorem \ref{harmonic}.

\begin{proof}[Proof the Theorem \ref{harmonic}]
Let $\omega$ be $L^p_f$ harmonic $1$-form on ${M},$ i.e.,
$$\triangle_f \omega=0, \int_{{M}}|\omega|^{p}e^{-f}d\mu<\infty.$$
We denote the dual vector field of $\omega$ by $\omega^\sharp.$ Applying the extended Bochner formula for $L^p_f$ harmonic $1$-form, we get
\begin{align}\label{012}
\Delta_f {{\left| \omega  \right|}^{2}}&=2{{\left| \nabla \omega  \right|}^{2}}+2\left\langle \Delta_f \omega ,\omega  \right\rangle +2\Ric_f\left( {{\omega }^{\sharp }},{{\omega }^{\sharp }} \right)\nonumber\\
&=2{{\left| \nabla \omega  \right|}^{2}}+2\Ric_f\left( {{\omega }^{\sharp }},{{\omega }^{\sharp }} \right).
\end{align}
Note that 
\begin{align}
{{\Delta }_{f}}{{\left| \omega \right|}^{2}}=2\left| \omega \right|{{\Delta }_{f}}\left| \omega \right|+2{{\left| \nabla \left| \omega \right| \right|}^{2}}\nonumber.
\end{align}
and the Bakry-Emery Ricci tensor of ${M},$ satisfies
\begin{align}
\Ric_f\left( {{\omega }^{\sharp }},{{\omega }^{\sharp }} \right)
&=-\left\langle {{A}^{2}}{{\omega }^{\sharp }},{{\omega }^{\sharp }}\nonumber \right\rangle.
\end{align}
This implies
$$\left| \omega  \right|{{\Delta }_{f}}\left| \omega  \right|={{\left| \nabla \omega  \right|}^{2}}-{{\left| \nabla \left| \omega  \right| \right|}^{2}}-\left\langle {{A}^{2}}{{\omega }^{\sharp }},{{\omega }^{\sharp }} \right\rangle .$$
Consequently,  by Kato inequality, we have
$$\left| \omega  \right|{{\Delta }_{f}}\left| \omega  \right|\ge -\left\langle {{A}^{2}}{{\omega }^{\sharp }},{{\omega }^{\sharp }} \right\rangle \ge -\left| {{A}^{2}}{{\omega }^{\sharp }} \right|\left| {{\omega }^{\sharp }} \right|\ge -\left| {{A}} \right|^2{{\left| \omega  \right|}^{2}}.$$
Now, let $\eta$ be a smooth compactly supported function on ${M}.$ By multiplying both sides
of the above inequality by $\eta^{2}|\omega|^{p-2}$  and then integrating the obtained result, we arrive at
\begin{align}\label{015}
\int_{{M} }{{{\eta }^{2}}{{\left| \omega  \right|}^{p-1}}{{\Delta }_{f}}\left| \omega  \right|{{e}^{-f}}d\mu }\ge -\int_{{M} }{\left| {{A}} \right|^{2}{{\left| \omega  \right|}^{p}}{{\eta }^{2}}{{e}^{-f}}d\mu }.
\end{align}
Since $\eta$ has compact support on ${M},$ by the Stokes theorem, we see that
\begin{align}
&\int_{{M} }{{{\eta }^{2}}{{\left| \omega \right|}^{p-1}}{{\Delta }_{f}}\left| \omega \right|{{e}^{-f}}d\mu }\nonumber\\
&\quad\quad=-\int_{{M} }{\left\langle \nabla \left( {{\eta }^{2}}{{\left| \omega \right|}^{p-1}} \right),\nabla \left| \omega \right| \right\rangle {{e}^{-f}}d\mu }\nonumber\\
&\quad\quad=-2\int_{{M} }{ {{\left| \omega \right|}^{p-1}}\left\langle \nabla \eta ,\nabla \left| \omega \right| \right\rangle \eta {{e}^{-f}}d\mu }-(p-1)\int_{{M} }{{{\eta }^{2}}{{\left| \omega \right|}^{p-2}}{{\left| \nabla \left| \omega \right| \right|}^{2}}{{e}^{-f}}d\mu }.\nonumber
\end{align}
This inequality and \eqref{015} implies
\begin{align}\label{017}
&(p-1)\int_{{M} }{{{\eta }^{2}}{{\left| \omega  \right|}^{p-2}}{{\left| \nabla \left| \omega  \right| \right|}^{2}}{{e}^{-f}}d\mu }\nonumber\\
&\quad\quad\le -2\int_{{M} }{{{\left| \omega  \right|}^{p-1}}\left\langle \nabla \eta ,\nabla \left| \omega  \right| \right\rangle \eta {{e}^{-f}}d\mu }+\int_{{M} }{{{\left| A \right|}^{2}}{{\left| \omega  \right|}^{p}}{{\eta }^{2}}{{e}^{-f}}d\mu }.
\end{align}
By Holder inequality and weighted Sobolev inequality, we have 
\begin{align}\label{018}
&\int_{{M} }{{{\left| A \right|}^{2}}{{\left| \omega  \right|}^{p}}{{\eta }^{2}}{{e}^{-f}}d\mu }\nonumber\\
&\quad\le\left(\int_{M}|A|^n\right)^\frac{2}{n}\left(\int_{M}(\eta|\omega|^\frac{p}{2})^\frac{2n}{n-2}\right)^\frac{n-2}{n}\nonumber\\
&\quad\le \left(\dfrac{2C(n)}{n-1}\right)^2\left\| A \right\|_{n,f}^{2}\int_{{M} }{{{\left| \nabla \left( \eta {{\left| \omega  \right|}^{\frac{p}{2}}} \right) \right|}^{2}}{{e}^{-f}}d\mu }\nonumber\\
&\quad=D_n\left\| A \right\|_{n,f}^{2}\int_{{M} }{\left( {{\left| \omega  \right|}^{p}}{{\left| \nabla \eta  \right|}^{2}}+p{{\left| \omega  \right|}^{p-1}}\left\langle \nabla \left| \omega  \right| ,\nabla \eta  \right\rangle \eta +\frac{{{p}^{2}}}{4}{{\left| \omega  \right|}^{p-2}}{{\eta }^{2}}{{\left| \nabla \left| \omega  \right|  \right|}^{2}} \right){{e}^{-f}}d\mu },
\end{align}
where $D_n= \left(\frac{2C(n)}{n-1}\right)^2.$ Using the Cauchy-Schwarz inequality and the inequality $xy\le \varepsilon {{x}^{2}}+\frac{y^2}{4\varepsilon }$ for any $\varepsilon>0,$ we see that
\begin{align}
p{{\left| \omega  \right|}^{p-1}}\left\langle \nabla \left| \omega  \right|,\nabla \eta  \right\rangle \eta
&\le p{{\left| \omega  \right|}^{p-1}}\left| \nabla \left| \omega  \right| \right|\left| \nabla \eta  \right|\left| \eta  \right|\nonumber\\
&=\frac{{{\left| \nabla \eta  \right|}^{2}}{{\left| \omega  \right|}^{p}}}{\varepsilon }+\frac{\varepsilon {{p}^{2}}}{4}{{\left| \omega  \right|}^{p-2}}{{\eta }^{2}}{{\left| \nabla \left| \omega  \right| \right|}^{2}}.\nonumber
\end{align}
This together with \eqref{018} implies
\begin{align}
&\int_{{M} }{{{\left| A \right|}^{2}}{{\left| \omega  \right|}^{p}}{{\eta }^{2}}{{e}^{-f}}d\mu }\nonumber\\
&\quad\le D_n\left\| A \right\|_{n,f}^{2}\left[ \left( 1+\frac{1}{\varepsilon } \right)\int_{{M} }{{{\left| \omega  \right|}^{p}}{{\left| \nabla \eta  \right|}^{2}}{{e}^{-f}}d\mu }+\frac{(1+\varepsilon ){{p}^{2}}}{4}\int_{{M} }{{{\left| \omega  \right|}^{p-2}}{{\eta }^{2}}{{\left| \nabla \left| \omega  \right| \right|}^{2}}{{e}^{-f}}d\mu } \right]\nonumber\\
&\quad=\left( 1+\frac{1}{\varepsilon } \right)D_n\left\| A \right\|_{n,f}^{2}\int_{{M} }{{{\left| \omega  \right|}^{p}}{{\left| \nabla \eta  \right|}^{2}}{{e}^{-f}}d\mu }\nonumber\\
&\quad\quad+\frac{(1+\varepsilon ){{p}^{2}}}{4}D_n\left\| A \right\|_{n,f}^{2}\int_{{M} }{{{\left| \omega  \right|}^{p-2}}{{\eta }^{2}}{{\left| \nabla \left| \omega  \right| \right|}^{2}}{{e}^{-f}}d\mu }.\label{020}
\end{align}
On the other hand, for any $\varepsilon>0,$ we have 
\begin{align}\label{021}
-2\int_{{M} }{{{\left| \omega  \right|}^{p-1}}\left\langle \nabla \eta ,\nabla \left| \omega  \right| \right\rangle \eta {{e}^{-f}}d\mu }
&\le 2\int_{{M} }{{{\left| \omega  \right|}^{p-1}}\left| \left\langle \nabla \eta ,\nabla \left| \omega  \right| \right\rangle  \right|\left| \eta  \right|{{e}^{-f}}d\mu }\nonumber\\
&\le \frac{1}{\varepsilon }\int_{{M} }{{{\left| \nabla \eta  \right|}^{2}}{{\left| \omega  \right|}^{p}}{{e}^{-f}}d\mu }+\varepsilon \int_{{M} }{{{\left| \omega  \right|}^{p-2}}{{\left| \nabla \left| \omega  \right| \right|}^{2}}{{\eta }^{2}}{{e}^{-f}}d\mu }.
\end{align}
Combining \eqref{017}, \eqref{020}, and \eqref{021}, we get 
\begin{align}
&\left[ p-1-\frac{{{p}^{2}}}{4}D_n\left\| A \right\|_{n,f}^{2}-\frac{\varepsilon {{p}^{2}}}{4}D_n\left\| A \right\|_{n+1,f}^{2}+\varepsilon  \right]\int_{{M} }{{{\eta }^{2}}{{\left| \omega  \right|}^{p-2}}{{\left| \nabla \left| \omega  \right| \right|}^{2}}{{e}^{-f}}d\mu }\nonumber\\
&\quad\quad\quad\le \left[ \left( 1+\frac{1}{\varepsilon } \right)D_n\left\| A \right\|_{n,f}^{2}+\frac{1}{\varepsilon } \right]\int_{{M} }{{{\left| \omega  \right|}^{p}}{{\left| \nabla \eta  \right|}^{2}}{{e}^{-f}}d\mu }.\nonumber
\end{align}
For a sufficiently small $\varepsilon>0,$ the above inequality implies that there is constant $C>0$ such that 
\begin{align}\label{23}
\int_{{M} }{{{\left| \omega  \right|}^{p-2}}{{\left| \nabla \left| \omega  \right| \right|}^{2}}{{e}^{-f}}{{\eta }^{2}}d\mu }\ \le C\int_{{M} }{{{\left| \omega  \right|}^{p}}{{\left| \nabla \eta  \right|}^{2}}{{e}^{-f}}d\mu },	
\end{align}
provided that 
$$p-1-\frac{{{p}^{2}}}{4}D_n\left\| A \right\|_{n,f}^{2}>0,$$
or equivalent 
$$\left\| A \right\|_{n,f}^{2} <  \frac{4\left( p-1 \right)}{{{p}^{2}}{{D}_{n}}}=\frac{\left( p-1 \right){{\left( n-1 \right)}^{2}}}{{{p}^{2}}C^2\left( n \right)}.$$
Let $o \in {M}$ be a fixed point and let $B_{r}(o)$ be the geodesic ball centered at $o$ with radius
$R.$ We choose $\eta$ to be a smooth function on ${M}$ such that $0\le \eta \le 1.$ Moreover, $\eta$ satisfies:

\begin{itemize}
\item [(i)] $\eta=1$ on $B_{R/ 2}(o)$ and $\eta=0$ outside $B_{R}(o)$;
\item [(ii)] $\left| \nabla \eta  \right|\le \frac{2}{R}$. 
\end{itemize}
Applying this test function $\eta$ to \eqref{23}, we get
\begin{equation}\label{de1}
\int_{{{B}_{R}}\left( o \right)}{{{\left| \omega  \right|}^{p-2}}{{\left| \nabla \left| \omega  \right| \right|}^{2}}{{e}^{-f}}d\mu }\ \le \frac{4C}{{{R}^{2}}}\int_{{{B}_{R}}\left( o \right)}{{{\left| \omega  \right|}^{p}}{{e}^{-f}}d\mu }.
\end{equation}
Letting $R$ tend to $\infty$ in the above inequality and note that $\omega \in L^f_p,$ we conclude that $\nabla\left| \omega  \right|=0,$ which shows that $\left| \omega  \right|$ is a constant. Moreover, since $\int_{{M}}|\omega|^{p}e^{-f}d\mu<\infty$ and the weighted volume of ${M}$ is infinite, we finally get $\omega =0.$ The proof is complete.
\end{proof}
Now, we note that if a Sobolev inequality holds true on ${M}$ every end of ${M}$ is non-$f$-parabolic, for example see \cite{IR17}. Therefore, we have the following corollary.
\begin{corollary}
Let $x: {M}^{n} \rightarrow \mathbb{R}^{n+1}$ be a smooth complete translator contained in the halfspace $\Pi_{v,a}=\{y\in\mathbb{R}^{n+1}: \left\langle y,v\right\rangle\geq a\}$, where $a\in\mathbb{R}, v\in\mathbb{R}^{n+1}$ are fixed  and $n\geq3$. Furthermore, assume that
$${{\left\| A \right\|}_{n,f}}\le \frac{ n-1}{2S(n)},$$
	where $S(n)$ is the constant as in Lemma \ref{sob}.
Then there are no nontrivial $L_f^2$ harmonic $1$-forms on ${M}$. In particular, ${M}$ has only one end. 
\end{corollary}
\begin{proof}Since every end of ${M}$ is non-$f$-parabolic, we can argue by contradiction to assume that ${M}$ has at least two ends. Then by Li-Tam theory \cite{litam}, there exists a non-constant $f$-harmonic function $u$ such that $\omega:=du$ satisfying $|\omega|\in L^2_f$. An application of Theorem \ref{harmonic} implies that $\omega=0$ or $u$ is constant. This is a contradiction. The proof is complete. 
\end{proof}

\section{Translators with a Sobolev inequality}
Suppose that ${M}$ satisfies the following Sobolev inequality
\begin{equation}\label{sobolev}
\left[\int_{M} u^{\frac{2(n+1)}{n-1}} \rho d \mu\right]^{\frac{n-1}{n+1}} \leq\left(\frac{2 C(n) n}{n-1}\right)^{2} \int_{M}|\nabla u|^{2} \rho d \mu
\end{equation}
for any $u$  that is a non-negative compactly supported $C^1$ function on $M$ and $C(n)$ is the Sobolev constant. In fact, the above inequality was proved in \cite{IR17}. However, the authors pointed out in \cite{IR19} that there is a gap in their proof of this inequality. Here, we assume that this inequality holds true. The Sobolev inequality \eqref{sobolev} was used by Kunikawa and Saito in \cite{KS19} to study the injectivity of the natural map between the first de Rham cohomology group with compact support, the reduced $L^2_f$ cohomology, and the space of $L^2_f$ $f$-harmonic $1$-forms. They proved that if $M$ supports the Sobolev inequality \eqref{sobolev} and admits a codimension one cycle which does not disconnect $M$ then the space of $L^2_f$ $f$-harmonic $1$-forms is non-trivial. 

Now, we apply the above Sobolev's inequality above to $u=|{\Phi}|^a \eta$. Then we have 
\begin{equation}
\begin{aligned}
\left[\int_{M} \left(|{\Phi}|^a \eta \right)^{\frac{2(n+1)}{n-1}} \rho \right]^{\frac{n-1}{n+1}} 
 &\leq\left(\frac{2 C(n) n}{n-1}\right)^{2} \int_{M}\left|\nabla  \left(|{\Phi}|^a \eta \right)\right|^{2} \rho \\
 &=\left(\frac{2 C(n) n}{n-1}\right)^{2}\left(\int_{M} a^2|\nabla| {\Phi}||^{2}|{\Phi}|^{2a-2} \eta^{2} \varrho\right.\\
&\left.\quad +\int_{M} 2a|{\Phi}|^{2a-1} \eta \nabla|{\Phi}| \cdot \nabla \eta \varrho+\int_{M}|{\Phi}|^{2a}|\nabla \eta|^{2} \varrho \right).\\
\end{aligned}
\end{equation} 
By the Cauchy inequalities, we obtain 
\begin{equation}
\begin{aligned}
\left[\int_{M} \left(|{\Phi}|^a \eta \right)^{\frac{2(n+1)}{n-1}} \rho \right]^{\frac{n-1}{n+1}}
 &\leq \left(\frac{2 C(n) n}{n-1}\right)^{2}\left(\int_{M} a^2|\nabla| {\Phi}||^{2}|{\Phi}|^{2a-2} \eta^{2} \varrho\right.\\
&\left. \quad+\int_{M} 2a|{\Phi}|^{2a-1} \eta \nabla|{\Phi}| \cdot \nabla \eta \varrho+\int_{M}|{\Phi}|^{2a}|\nabla \eta|^{2} \varrho \right)\\
&\leq \left(\frac{2 C(n) n}{n-1}\right)^{2}\left((1+\delta)\int_{M} a^2|\nabla| {\Phi}||^{2}|{\Phi}|^{2a-2} \eta^{2} \varrho\right.\\
&\left.\quad +\left(1+\frac{1}{\delta} \right) \int_{M}|{\Phi}|^{2a}|\nabla \eta|^{2} \varrho \right).
\end{aligned}
\end{equation} 
 Apply \eqref{eq:eq6} and keep in mind that right now $\iota=2$. For $0< \varepsilon < a -\frac{1}{2}$, we have 
 \begin{equation}\label{eq:eq13}
 \begin{aligned}
\kappa_{1}^{-1} \left[\int_{M} \left(|{\Phi}|^a \eta \right)^{\frac{2(n+1)}{n-1}} \rho \right]^{\frac{n-1}{n+1}} &\leq \left\{\frac{a^2(1+\delta)}{4\left(a-\frac{1}{2}-\varepsilon\right)}\left(2 \int_{M}|{\Phi}|^{2a+2} \eta^{2} \varrho\right.\right.\\
&\quad\left.+\frac{2}{n} \int_{M}|{\Phi}|^{2a}|H|^{2} \eta^{2} \varrho+\frac{1}{\varepsilon} \int_{M}|{\Phi}|^{2a}|\nabla \eta|^{2} \varrho\right) \\
&\quad\left.+\left(1+\frac{1}{\delta}\right) \int_{M}|{\Phi}|^{2a}|\nabla \eta|^{2} \varrho\right\},
\end{aligned}
 \end{equation}
where $\kappa_1=\left(\frac{2C(n)n}{n-1}\right)^2$. 

Using the fact that $|A|^2=|{\Phi}|^2+\frac{1}{n}|H|^2$, we can rewrite \eqref{eq:eq13} as 
 \begin{equation}\label{eq:eq14}
  \begin{aligned}
\kappa_{1}^{-1} \left[\int_{M} \left(|{\Phi}|^a \eta \right)^{\frac{2(n+1)}{n-1}} \rho \right]^{\frac{n-1}{n+1}} &\leq \frac{a^2(1+\delta)}{2\left(a-\frac{1}{2}-\varepsilon\right)}\int_{M}|{\Phi}|^{2a} |A|^2 \eta^{2} \varrho\\
&\quad+ \tilde{C}(n,a,\delta,\varepsilon) \int_{M}|{\Phi}|^{2a}|\nabla \eta|^{2} \varrho ,
\end{aligned}
 \end{equation}
 where $\tilde{C}(n,a,\delta,\varepsilon)$ is an explicit positive constant depending on $n,a,\delta, \varepsilon$. 
 
 By the Holder inequality, we have that 
 \begin{equation}\label{eq:eq15}
 \begin{aligned}
\int_{M}|{\Phi}|^{2a} |A|^2 \eta^{2} \varrho & \leq\left(\int_{M}|A|^{2 \cdot \frac{n+1}{2}} \varrho\right)^{\frac{2}{n+1}} \cdot\left(\int_{M}\left(|{\Phi}|^{2a} \eta^{2} \right)^{\frac{n+1}{n-1}} \varrho\right)^{\frac{n-1}{n+1}}\\
										 & = \left(\int_{M}|A|^{n+1}\right)^{\frac{2}{n+1}} \cdot\left(\int_{M}\left(|{\Phi}|^{a} \eta \right)^{\frac{2(n+1)}{n-1}} \varrho\right)^{\frac{n-1}{n+1}}.
\end{aligned}
 \end{equation}
 Our goal is to decrease the number of conditions in theorem \ref{th:th12}, only one condition instead of two as in Theorem \ref{th:th12}, so we should choose $a=\frac{n+1}{2}$. For that reason, combining \eqref{eq:eq14} and \eqref{eq:eq15}, we have
 $$\begin{aligned}
\kappa_{1}^{-1} \left[\int_{M} \left(|{\Phi}|^{\frac{n+1}{2}} \eta \right)^{\frac{2(n+1)}{n-1}} \rho \right]^{\frac{n-1}{n+1}}&\leq \frac{(n+1)^2(1+\delta)}{8\left(\frac{n+1}{2}-\frac{1}{2}-\varepsilon\right)} \left(\int_{M}|A|^{n+1}\right)^{\frac{2}{n+1}} \cdot\left(\int_{M}\left(|{\Phi}|^{\frac{n+1}{2}} \eta \varrho \right)^{\frac{2(n+1)}{n-1}} \varrho\right)^{\frac{n-1}{n+1}}\\
&\quad+ \tilde{C}(n,\delta,\varepsilon) \int_{M}|{\Phi}|^{n+1}|\nabla \eta|^{2} \varrho ,
\end{aligned}$$
Put $$K_1(n,\varepsilon,\delta) = \sqrt{\frac{8\left(\frac{n+1}{2}-\frac{1}{2}-\varepsilon\right)}{(n+1)^2(1+\delta)\kappa_1}},$$
and 
$$K_1(n) = \sup_{\delta>0, \,0< \varepsilon <a-\frac{1}{2}} K_1(n,\varepsilon,\delta) = \sqrt{\frac{(n-1)^2}{C(n)^2(n+1)^2 n^2}}.$$
Applying the argument as in the proof of Theorem \ref{th:th12}, we have the following result. 
\begin{theorem}
Let $M^n (n \geq 3)$ be a smooth complete translating soliton in
the Euclidean space $\mathbb{R}^{n+1}$ with Sobolev inequality \eqref{sobolev}. If the second fundamental form $A$ of $M$
satisfies
$$\left(\int_{M} |A|^{n+1} \varrho \right)^{\frac{1}{n+1}}<K_1(n),$$
where $K_1(n)$ is defined as above , then $M$ is a
hyperplane.
\end{theorem}
\section*{Conflict of interest statement}
On behalf of all authors, the corresponding author states that there is no conflict of interest.

\section*{Acknowledgement} 
This work was initial started during a stay of the second author at VIASM. He would like to thank the staff there for hospitality and support.

\bigskip
\address{{\it Ha Tuan Dung}\\
Faculty of Mathematics \\
Hanoi Pedagogical University No. 2\\ 
Xuan Hoa, Vinh Phuc, Vietnam}
{hatuandung@hpu2.edu.vn}

\address{ {\it Nguyen Thac Dung}\\
Faculty of Mathematics - Mechanics - Informatics \\
Vietnam National University, University of Science, Hanoi \\
Hanoi, Viet Nam and \\
Thang Long Institute of Mathematics and Applied Sciences (TIMAS)\\
Thang Long Univeristy\\ 
Nghiem Xuan Yem, Hoang Mai\\
Hanoi, Vietnam
}
{dungmath@gmail.com or dungmath@vnu.edu.vn}
\address{{\it Tran Quang Huy}\\
Faculty of Mathematics - Mechanics - Informatics\\
Vietnam National University, University of Science, Hanoi \\
Hanoi, Viet Nam}
{\\tranquanghuy11061998@gmail.com}


\begin{thebibliography}{99}
\bibitem{CS80} D. Fischer-Colbrie and R. Schoen, \emph{The structure of complete stable minimal surfaces in 3-manifolds of nonnegative scalar curvature}, Comm. Pure Appl. Math. \textbf{33} (2) (1980) 199-211.

\bibitem{Guang19} Q. Guang, \emph{Volume growth, entropy and stability for translating solitons}, Comm. Anal. Geom. \textbf{27} (2019), no. 1, 47-72.

\bibitem{Has} R. Haslhofer, \emph{Uniqueness of the bowl soliton}, Geom. Topol. \textbf{19} (4) (2015), 2393-2406.

\bibitem{HS74} D. Hoffman and J. Spruck, \emph{Sobolev and isoperimetric inequalities for Riemannian submanifolds}, Comm. Pure Appl. Math. \textbf{27} (1974), 715–727.

\bibitem{HS99} G. Huisken and C. Sinestrari, \emph{Convexity estimates for mean curvature flow and singularities of mean
convex surfaces}, Acta Math. \textbf{183} (1) (1999), 45-70.

\bibitem{IR17} D. Impera and M. Rimoldi, \emph{Rigidity results and topology at infinity of translating solitons of the mean curvature flow}, Commun. Contemp. Math. \textbf{19} (2017), no. 6, 1750002.

\bibitem{IR19} D. Impera and M. Rimoldi, \emph{Quantitative index bounds for translators via topology}, Math. Zeits. \textbf{292} (2019), no. 1-2, 513-527.

\bibitem{KP17} D.H. Kim and J.C. Pyo, \emph{Translating solitons foliated by spheres}, Internat. Jour. Math. \textbf{28} (2017), no. 1, 1750006.

\bibitem{KS19} K. Kunikawa and S. Saito, \emph{Remarks on topology of stable translating solitons}, Geom. Dedicata \textbf{202} (2019), 1-8.

\bibitem{litam} P. Li and L.F. Tam, \emph{Harmonic functions and the structure of complete manifolds}, J. Diff. Geom. \textbf{35} (2) (1992), 359-383.

\bibitem{Ma} L. Ma and V. Miquel, \emph{ Bernstein theorem for translating solitons of hypersurfaces}, Manuscripta Math. \textbf{162} (2020), 115-132.

\bibitem{MS73}  J. H. Michael and L. M. Simon, \emph{Sobolev and mean-value inequalities on generalized submanifolds of $\mathbb{R}^n$}, Comm. Pure Appl. Math. \textbf{26} (1973), 361-379.

\bibitem{Nguyen22} X. H. Nguyen, \emph{Doubly periodic self-translating surfaces for the mean curvature flow}, Geom. Dedicata, \textbf{174} (2015) 177-185.

\bibitem{Nguyen23} X. H. Nguyen, \emph{Complete embedded self-translating surfaces under mean curvature flow}, Jour. Geom. Anal. \textbf{23} (2013) 1379-1426.


\bibitem{42} L. Shahriyari, Translating Graphs by Mean Curvature Flow (ProQuest LLC, Ann Arbor, MI, 2013); Ph.D. thesis, The Johns Hopkins University.

\bibitem{421} L. Shahriyari,\emph{Translating graphs by mean curvature flow,} Geometriae Dedicata, \textbf{175}(1) (2014), 57-64.

\bibitem{Smoczyk} K. Smoczyk, \emph{A relation between mean curvature flow solitons and minimal submanifolds}, Math. Nachr. \textbf{229} (2001) 175-186
 
\bibitem{SSY} R. Schoen, L. Simon, S.Y. Yau, \emph{Curvature estimates for minimal hypersurfaces},
Acta Math. \textbf{134} (1975), 275-288.

\bibitem{WXZ16} H. J. Wang, H. W. Xu, and E. T. Zhao, \emph{A global pinching theorem for complete translating solitons of mean curvature flow}, Pure Appl. Math. Q. \textbf{12} (2016), no. 4, 603-619.

\bibitem{Wang11} X.J. Wang, \emph{Convex solutions to the mean curvature flow}, Ann. of Math. (2), \textbf{173} (3) (2011), 1185-1239.

\bibitem{Xin15} Y. L. Xin, \emph{Translating solitons of the mean curvature flow}, Calc. Var. Partial
 Differ. Equations \textbf{54} (2015), 1995-2016.
 
\bibitem{XG07} H. W. Xu and J. R. Gu, \emph{A general gap theorem for submanifolds with parallel mean curvature in $\mathbb{R}^{n+m}$}, Comm. Anal. Geom., \textbf{15} (2007), 175-194





\end{thebibliography}
\end{document}